\documentclass[12pt,a4paper]{article}
\usepackage{url}
\usepackage{amsmath, amssymb, amsthm}
\usepackage{xcolor}
\usepackage{graphicx}
\usepackage{enumerate}
\usepackage{tikz}

\DeclareMathOperator{\arccosh}{arccosh}

\newtheorem{definition}{Definition}[section]
\newtheorem{theorem}[definition]{Theorem}
\newtheorem{lemma}[definition]{Lemma}
\newtheorem{proposition}[definition]{Proposition}
\newtheorem{corollary}[definition]{Corollary}
\newtheorem{remark}[definition]{Remark}
\newtheorem{assumption}[definition]{Assumption}
\newtheorem{notation}[definition]{Notation}
\theoremstyle{definition}

\title{Anomalous Recurrence Properties of Markov Chains on Manifolds of Negative Curvature}
\author{John Armstrong\footnote{King's College London} \: and Tim King\footnote{This work was supported by the Engineering and Physical Sciences Research Council [EP/L015234/1]. The EPSRC Centre for Doctoral Training in Geometry and Number Theory (The London School of Geometry and Number Theory), University College London. The second author is also a member of King's College London and thanks the same for its support.}}

\begin{document}

\maketitle

\begin{abstract}
We present a recurrence-transience classification for discrete-time Markov chains on manifolds with negative curvature. Our classification depends only on geometric quantities associated to the increments of the chain, defined via the Riemannian exponential map. We deduce that there exist Markov chains on a large class of such manifolds which are both recurrent and have zero average drift at every point. We give an explicit example of such a chain on hyperbolic space of arbitrary dimension, and also on a stochastically incomplete manifold. We also prove that such recurrent chains cannot be uniformly elliptic, in contrast with the Euclidean case.
\end{abstract}

\section{Introduction}

It is a classical result \cite{ichihara1982curvature} that Brownian motion in hyperbolic space is transient in dimensions two and higher, in contrast to the Euclidean case \cite{kakutani1944131}, where it is recurrent in dimension two. In this paper, we study more general random walks on negatively curved manifolds. We focus our attention on cases where the process respects the geometry of the manifold. Specifically, we consider discrete-time Markov processes which have martingale-like properties. To define a martingale on a manifold, one needs some geometric structure. For our purposes, we will be interested in the processes where, in the chart induced by the Riemannian exponential map, each increment has zero mean. Such processes are called zero-drift processes.

One might anticipate that the qualitative long-term behaviour of a process, such as recurrence, will be determined by its drift properties alone. However, in Euclidean space this is known to be false. In \cite{georgiou2016anomalous}, the authors give examples of recurrent zero-drift chains in $\mathbb{R}^d$ for arbitrarily large fixed $d$, the increments of which have a finite covariance matrix at every point. They further obtain a recurrence-transience classification result of Lamperti type \cite{lamperti1960criteria}, requiring only local information obtained from the aforementioned covariance matrices. Moreover, they give examples of recurrent chains which are uniformly elliptic, meaning that for any fixed direction, there is a probability of at least $\epsilon$ that the chain will move a distance at least $\epsilon$ in that direction (see Section \ref{section:assumptions} for a precise definition).

Existing results on the recurrence and transience of Brownian motion on a manifold (see e.g. \cite[Theorem 4.4.12]{hsu2002stochastic} or \cite[Theorem 1.1]{shiozawa2017escape}) suggest that the class of recurrent chains in the negative curvature case is likely to be qualitatively different from that in the Euclidean case. A somewhat striking manifestation of this is the existence of \textit{stochastically incomplete} manifolds, where Brownian motion is not merely transient but explodes in finite time \cite{grigor1999analytic}. The qualitative difference motivates our consideration of what examples of the type found in \cite{georgiou2016anomalous} exist in the negative curvature case.

As in \cite{georgiou2016anomalous} we find a Lamperti-type criterion which, in certain situations, allows the use of local information derived from the increments (although not necessarily the covariance matrices) to decide whether a given manifold-valued Markov chain is recurrent. As a consequence, we give an example of a zero-drift, recurrent Markov chain on a stochastically incomplete manifold. In contrast to the Euclidean case, we deduce from our criterion that recurrent zero-drift walks cannot be uniformly elliptic. We also explain quantitatively the extent to which uniform ellipticity must fail, in terms of the asymptotic behaviour of the curvature of the manifold, if a zero-drift chain is to be recurrent. Another contrast we observe is that in Euclidean space it is possible to give a simple recurrence criterion using the growth of quantities calculated from the covariance matrices. We give an example (Proposition \ref{prop:stillnotrecurrent}) to show that the corresponding results do not hold in hyperbolic space for any polynomial growth condition.

The paper is constructed as follows. Section 2 gives a precise description of our model and states a recurrence-transience criterion for constant curvature manifolds. Section 3 proves this result, and Section \ref{section:nonconstantcurvature} explains how the result may still be applied even if the curvature is not constant. Section \ref{section:nonconfinement} gives a sufficient local condition for a chain to not be trapped in a finite region. Section \ref{section:eucvshyp} compares the Euclidean and hyperbolic cases in more detail, and Section \ref{section:examples} gives some examples.

\section{Model and Main Results}
\label{section:assumptions}
Throughout this paper, $M$ denotes a fixed $d$-dimensional Riemannian manifold and $X=(X_n)_{n \in \mathbb{N}}$ denotes a discrete-time time-homogeneous Markov chain with state space $M$ and underlying probability space $\Omega$. Measurability is defined via the Borel sigma algebra on $M$. We make the following assumptions on $M$; the reader unfamiliar with the differential geometric concepts below may consult (e.g.) \cite{carmo1992riemannian}.
\begin{assumption}
\label{assumptionnegativecurvature}
$M$ is complete, simply connected and of everywhere nonpositive sectional curvature.
\end{assumption}
To define the sectional curvature of $M$ at a point $p$, it is necessary to choose a plane $\Pi$ in the tangent space $T_p M$. Throughout, bounds on sectional curvature are assumed to hold for all possible choices of $p$ and $\Pi$. As explained in \cite[Lemma 2.1.4]{jost2012nonpositive}, Assumption \ref{assumptionnegativecurvature} implies that for every point $p \in M$, the exponential map $\exp_p: T_pM \rightarrow M$ is a diffeomorphism. Assumption \ref{assumptionnegativecurvature} also implies that if $p,q$ are distinct points in $M$, then there is a unique geodesic segment in $M$ joining $p$ and $q$. We define the distance $\text{Dist}_M(p,q)$ to be the Riemannian length of this geodesic segment.

We make the following assumptions on the chain. 

\begin{assumption}
\label{assumptiondtotmoments}
There exists $p>2$ and $B \in \mathbb{R}$ such that \[\mathbb{E}[\text{Dist}_M(X_{n},X_{n+1})^p \mid X_n=x] \leq B\] for all $x \in M$ and for some (equivalently for all) $n \in \mathbb{N}$.
\end{assumption}
\begin{assumption}
\label{assumptionnotconfined}
$\limsup_{n \rightarrow \infty} \text{Dist}_M (X_{n},x) = \infty$ almost surely for some (equivalently for all) $x \in M$.
\end{assumption}

One motivation behind these assumptions is that they disallow certain uninteresting cases. For example, without Assumption \ref{assumptiondtotmoments} we could give the chain a probability of $10^{-6}$ (say) of moving directly to some fixed point $p$ on every step, which would trivially give point-recurrence at $p$. The non-confinement assumption \ref{assumptionnotconfined}, despite being global in nature, is easier to check in practice than it might appear. In the Euclidean case, Proposition 2.1 of \cite{georgiou2016anomalous} gives some local conditions which imply Assumption \ref{assumptionnotconfined}. We give similar local criteria for non-confinement in Section \ref{section:nonconfinement}.

From now on, we assume that an origin $O \in M$ has been chosen. We define the \textit{radial distance process} for $X$ (with respect to $O$) to be
\begin{equation*}
    R_n := \text{Dist}_M(O,X_n)
\end{equation*}
We will show that, given our assumptions, $X$ must behave in one of two (a priori non-exhaustive) ways.
\begin{definition} 
Let $O \in M$ be a point. A Markov chain $X$ in $M$ is called: \\
(i) \textit{O-recurrent} if there is some constant $r_0$ such that $\liminf_{n \rightarrow \infty} R_n \leq r_0$ a.s. \\
(ii) \textit{O-transient} if $\lim R_n = \infty$ a.s.
\end{definition}

%Enumerate not used as it takes up space causing inconvenient page-splitting
%\begin{definition} 
%Let $O \in M$ be a point. A Markov chain $X$ in $M$ is called:
%\begin{enumerate}[(i)]
%    \item \textit{O-recurrent} if there is some $r_0>0$ such that $\liminf_{n %\rightarrow \infty} R_n \leq r_0$ a.s.
%    \item \textit{O-transient} if $\lim R_n = \infty$ a.s
%\end{enumerate}
%\end{definition}

It is immediate from the triangle inequality that if a walk is $O$-transient for some $O$ then it is $O$-transient for every $O$. The usual definition of recurrence on manifolds requires that, if $N$ is \textit{any} open set in $M$, then it is almost surely the case that $X_n \in N$ infinitely often \cite{grigor1999analytic}. This is therefore a stronger condition than that the walk be $O$-recurrent for every $O \in M$. We are predominantly interested in the weaker condition because it allows us to avoid technicalities concerning irreducibility (note also that the state space $M$ is uncountable). Nevertheless, in the appendix we outline how, in applications, recurrence in the usual sense may be established.

Since the exponential map is a diffeomorphism at every point, we may define
\begin{equation}
    \label{eq:vdef}
    v_n := \exp^{-1}_{X_n}(X_{n+1})
\end{equation}
Under Assumption \ref{assumptionnegativecurvature}, the tangent bundle $TM:=\sqcup_{p \in M} T_p M$ is diffeomorphic to $M \times \mathbb{R}^d$, and so $(X_n,v_n)_{n \in \mathbb{N}}$ is a Markov chain with state space $TM$, with the property that $X_{n+1} = \exp_{X_n}(v_n)$ for all $n$. The process $(X_n,v_n)_{n \in \mathbb{N}}$ has appeared in the literature under the name \textit{geodesic random walk}. The benefit of introducing the geodesic random walk is that, even though $M$ is not Euclidean, we can make use of Euclidean techniques using the fact that $\mathbb{R}^d$ is a vector space, together with the diffeomorphism $TM \simeq M \times \mathbb{R}^d$. For example, Jorgensen \cite{jorgensen1975central} proves an invariance principle for the geodesic random walk by rescaling both time and the $v_n$ in an appropriate manner. More recently Kraaij, Redig, and Versendaal have considered the large deviations of such walks \cite{kraaij2019classical}.

\begin{definition}
A chain $X$ on $M$ is called \textit{zero drift} if
\[ \mathbb{E}[\exp^{-1}_{X_n}(X_{n+1}) \mid \mathcal{F}_n]=0 \]
almost surely for all $n \in \mathbb{N}$, where the conditional expectation is defined using the vector space structure of $T_{X_n}M$, and $\mathcal{F}_n = \sigma(X_1,X_2,\dots,X_n)$.
\end{definition}

 Zero drift chains on $M$ are closely related to the concept of martingales on $M$, and indeed the two are equivalent when $M=\mathbb{R}^n$. To say what a martingale is in the more general case, one must define the notion of conditional expectation on a nonlinear space. Different methods for doing this appear in the literature -- see e.g. \cite{sturm2002nonlinear} for a development of the theory of martingales on general negatively curved spaces. The theory is more involved than in the Euclidean case. For example, the tower law $\mathbb{E}[X_{n+m} \mid \mathcal{F}_n]=X_n$ for all $m \in \mathbb{N}$ does not automatically hold. Our notion of zero-drift chains has been considered before, albeit phrased in terms of barycentres \cite{emery1991barycentre}.

For a point $x \in M$, let $\langle \cdot,\cdot \rangle_x$ denote the inner product induced on $T_xM$. Recall that Assumption \ref{assumptionnegativecurvature} implies that, provided $x \neq O$, there is a unique unit speed geodesic $\gamma: \mathbb{R} \rightarrow M$ going through $O$ such that $\gamma(0)=x$. Consequently, there is a distinguished vector $e_{\text{rad}}(x) \in T_x M$ such that $\exp_x(t e_{\text{rad}})=\gamma(t)$ for all $t$. Since $\gamma$ has unit speed, $\langle e_{\text{rad}}, e_{\text{rad}} \rangle_x=1$. Define the random variables $d_{\text{rad}}^n$ and $d_{\text{tot}}^n$ by
\begin{equation}
    \label{eq:dtotdef}
    d_{\text{tot}}^n := \langle v_n, v_n \rangle_{X_n} = \text{Dist}_M(X_n,X_{n+1})
\end{equation}
\begin{equation}
    \label{eq:draddef}
    d_{\text{rad}}^n :=
    \begin{cases} 
    -\langle v_n, e_{\text{rad}}(X_n) \rangle_{X_n} & \text{ if } X_n \neq O \\
    d_{\text{tot}}^n & \text{ if } X_n=O
    \end{cases}
\end{equation}

Figure 1 shows the geometric interpretation of the objects defined above. Informally, one should consider $v_n$ to point towards the origin if $d_{\text{rad}}^n$ is negative, although the statements $d_{\text{rad}}^n<0$ and $\text{Dist}_M(X_{n+1},O) < \text{Dist}_M(X_n,O)$ are not equivalent.
% NB: Figure referenced manually!

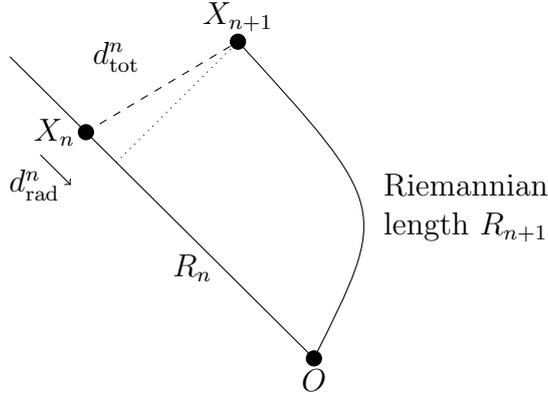
\begin{figure}[ht]
\label{figure:exponential}
\begin{center}
\begin{tikzpicture}
\draw (-4,4) -- (0,0);
\filldraw (0,0) circle (0.1cm) node[anchor = north] {$O$};
\filldraw (-3,3) circle (0.1cm) node[anchor = east] {$X_{n}$};
\filldraw (-1,4.2) circle (0.1cm) node[anchor = south] {$X_{n+1}$};
\draw[dashed] (-3,3) -- (-1,4.2);
\draw[dotted] (-2.6,2.6) -- (-1,4.2);
\draw[->] (-3.6,2.7) -- (-3.2,2.3) node[anchor = east]{$d_{\text{rad}}^n$};
\draw (0,0) .. controls (1,2)  .. (-1,4.2); 
\node[align=left] at (2,2) {Riemannian \\ length $R_{n+1}$};
\node[] at (-2.6,4) {$d_{\text{tot}}^n$};
\node[] at (-1.6,1.2) {$R_n$};
\end{tikzpicture}
\end{center}
\caption{Schematic showing $d_{\text{rad}}$ and $d_{\text{tot}}$ in the chart induced by the exponential map at $X_n$. In this example, $d_{\text{rad}}$ has negative sign. See Proposition \ref{prop:incrementexact} for a calculation of $R_{n+1}$.}
\end{figure}

\begin{notation}
\label{notation:markov} Throughout, we write expressions such as $\mathbb{E}_x[d_{\text{rad}}]$ to mean $\mathbb{E}[d^n_{rad} \mid X_n = x]$. This makes sense because, by the Markov and time-homogeneous properties of $X$, the latter expression depends only on $x$, not on $n$. We also write, for example, $\Delta R_n$ to mean $R_{n+1}-R_{n}$. We do not abbreviate expressions such as $\mathbb{E}[\Delta R_n \mid R_n=r]$ further because $R_n$ is not, in general, Markov. We also use $(\phi^n)_{n \in \mathbb{N}}$ as shorthand for the real-valued random variables defined by
\begin{equation}
    \label{eq:phidef}
    \phi^n = \begin{cases} 
    \frac{d^n_{\text{rad}}}{d^n_{\text{tot}}} & \text{ if } d^n_{\text{tot}} > 0 \\ 0 & \text{ if } d^n_{\text{tot}} = 0. \end{cases}
\end{equation}
Again, we sometimes omit the letter $n$ and just write $\phi$ where appropriate.
\end{notation}

Theorem \ref{thm:mainresultA}, stated below, is the most basic version of our main result; a recurrence-transience criterion for constant curvature manifolds. Although our main interest is in zero-drift chains, this result does not require the chain to be zero-drift. We stress that, in applications, this result may still be useful even if $M$ does not have constant curvature. See Section \ref{section:nonconstantcurvature}, and in particular Theorem \ref{thm:mainresultB}, for an analogue of Theorem \ref{thm:mainresultA} for pinched curvature manifolds.

\begin{theorem}
\label{thm:mainresultA}
Let $M$ be a manifold (with origin $O$) of constant curvature $-k^2$ for some $k>0$. Let $(X_n: n \in \mathbb{N})$ be a Markov chain in $M$. Assume that \ref{assumptiondtotmoments} and \ref{assumptionnotconfined} both hold. For $i=1,2$, let
\begin{equation*} 
\begin{split}
    S(r):= \{ x \in M : \text{Dist}_M(O,x)=r \} \\
\underline{\nu}_i(r)=\inf_{x \in S(r)} \frac{1}{k^i} \mathbb{E}_x[\log^i(\cosh(k d_{\text{tot}})+\phi \sinh(k d_{\text{tot}}))] \\
\overline{\nu}_i(r)=\sup_{x \in S(r)} \frac{1}{k^i} \mathbb{E}_x[\log^i(\cosh(k d_{\text{tot}})+\phi \sinh(k d_{\text{tot}}))] \\
\end{split}
\end{equation*}
where $\phi$ is as defined in Equation \eqref{eq:phidef}. Then \\
(i) Suppose that
\[ \limsup_{r \rightarrow \infty} \overline{\nu_2}(r)<\infty, \text{ and } \liminf(2r \underline{\nu}_1(r)-\overline{\nu}_2(r))>0  \]
Then $R_n \rightarrow \infty$ almost surely and the chain is $O$-transient. \\
(ii) Suppose instead that $\liminf_{r \rightarrow \infty} \underline{\nu_2}(r)>0$ and that there exist $r_0 \geq 0$, $\theta>0$ such that
\[ 2r\overline{\nu_1}(r) \leq \left(1+\frac{1-\theta}{\log r}\right)\underline{\nu_2}(r) \hspace{10pt} \forall r \geq r_0 \]
then there exists $b \geq 0$ such that $\liminf_{n \rightarrow \infty} R_n \leq b$ almost surely
and the chain is $O$-recurrent.
\end{theorem}

Recall from \cite{georgiou2016anomalous} that a chain $(X_n)_{n \in \mathbb{N}}$ is called \textit{uniformly elliptic} if, for some $\epsilon>0$,
\begin{equation}
\label{eq:uniformlyelliptic}
    \mathbb{P}[\langle \exp^{-1}_{X_n}(X_{n+1}), w \rangle \geq \epsilon] \geq \epsilon
\end{equation}
for all unit vectors $w \in T_{X_n}M$. In Euclidean space, zero-drift recurrent uniformly elliptic chains exist. By contrast, in Section \ref{section:eucvshyp}, we derive the following consequence of Theorem \ref{thm:mainresultA}. 

\begin{theorem}
\label{corol:uniformlyelliptic}
Let $M$ satisfy Assumption \ref{assumptionnegativecurvature} and assume in addition that there exists $k>0$ such that the sectional curvature of $M$ is at most $-k^2$ at every point. Let $X$ be a Markov chain on $M$ satisfying Assumptions \ref{assumptiondtotmoments} and \ref{assumptionnotconfined}. If $X$ is uniformly elliptic and of zero drift, then $X$ is transient.
\end{theorem}

\section{Geometric Estimates and Proof of Theorem \ref{thm:mainresultA}}
\label{section:geometricestimates}
Let $M$ be a manifold with origin $O \in M$. Let $(X_n: n \in \mathbb{N})$ be a Markov chain in $M$, and let $\mathcal{F}_n=\sigma(X_1,\dots,X_n)$. Then $R_n$ is adapted to $\mathcal{F}_n$. Our strategy for proving Theorem \ref{thm:mainresultA} is to estimate $\mathbb{E}[(\Delta R_n)^i \: | \: \mathcal{F}_n]$ for $i=1,2$, and then use the following Lamperti-type result, found in \cite[Chapter 3, pp.114-115]{menshikov2016non}. (Recall from Section \ref{section:assumptions} that $R_n:=\text{Dist}_M(O,X_n)$ and that $\Delta R_n := R_{n+1}-R_{n}$.)

\begin{theorem} \label{lyapunov} Let $(R_n: n \in \mathbb{N})$ be a stochastic process adapted to some filtration $(\mathcal{F}_n: n \in \mathbb{N})$ taking values in $\mathbb{R}_{\geq 0}$. Assume that $\limsup_{n \rightarrow \infty} R_n = \infty$ a.s. and that there exist $p>2$, $B>0$ such that $\mathbb{E}[(\Delta R_n)^p \mid \mathcal{F}_n] \leq B$ a.s. for all $n$. Suppose that we are given, for $i=1$ and $i=2$, Borel functions $\underline{\mu}_i(r), \overline{\mu}_i(r): \mathbb{R}_{\geq 0} \rightarrow \mathbb{R}$ such that
\begin{equation}
\label{eq:mubounds}
    \underline{\mu}_i(R_n) \leq \mathbb{E}[(\Delta R_n)^i \: | \: \mathcal{F}_n] \leq \overline{\mu}_i(R_n)
\end{equation}
almost surely for all $n$. Then\\
(i) `Transience': Suppose that
\[ \limsup_{r \rightarrow \infty} \overline{\mu_2}(r)<\infty, \text{ and } \liminf_{r \rightarrow \infty} (2r \underline{\mu}_1(r)-\overline{\mu}_2(r))>0  \]
Then $R_n \rightarrow \infty$ almost surely. \\
(ii) `Recurrence': Suppose instead that $\liminf_{r \rightarrow \infty} \underline{\mu_2}(r)>0$ and that there exist $r_0 \geq 0$, $\eta>0$ such that
\[ 2r\overline{\mu_1}(r) \leq \left(1+\frac{1-\eta}{\log r}\right)\underline{\mu_2}(r) \hspace{10pt} \forall r \geq r_0 \]
then there exists $b \geq 0$ such that $\liminf_{n \rightarrow \infty} R_n \leq b$ almost surely.
\end{theorem}

\begin{remark}
    Since $X$ is Markov we have\begin{equation*}
    \mathbb{E}[\Delta R_n \mid \mathcal{F}_n]=\mathbb{E}[\Delta R_n \mid X_n] \leq \sup_{x \in S(R_n)} \mathbb{E}[\Delta R_n \mid X_n = x]
\end{equation*}
and therefore it suffices for  $\underline{\mu}$ and $\overline{\mu}$ to satisfy
\begin{equation}
\label{eq:mucondition}
\begin{split}
     \underline{\mu}_i(r) \leq \inf_{x \in S(r)} \mathbb{E}_x[(\Delta R)^i]   \\
      \overline{\mu}_i(r) \geq \sup_{x \in S(r)} \mathbb{E}_x[(\Delta R)^i]
\end{split}
\end{equation}
where we recall that $S(r)= \{ x \in M : \text{Dist}_M(O,x)=r \}$.
\end{remark}

We give an exact expression for the radial increment in terms of $d_{\text{tot}}$, $d_{\text{rad}}$, and the current location of the chain. We defer the proof to the appendix, the strategy being to select an appropriate model of hyperbolic space (we use the Lorentz model), and then proceed by direct calculation. 
\begin{proposition}
\label{prop:incrementexact}
Let $M$ be a manifold of constant curvature $-k^2$. Choose an origin $O$ in $M$, and take a point $X_n \in M$, such that $\text{Dist}_M(O,X_n)=R$. Choose a vector $v \in T_{X_n} M$ of length $d_{\text{tot}}$ and radial length $d_{\text{rad}}$, and let $X_{n+1}=\exp_{X_n}(v)$. Then
\begin{equation}
    \text{Dist}_M(O,X_{n+1})=\frac{1}{k}\arccosh\left( \cosh{kR}\cosh{kd_{\text{tot}}}+\phi \sinh{kR}\sinh{kd_{\text{tot}}} \right)
\end{equation}
\end{proposition}

\begin{proposition} \label{prop:incrementapproximation}
Let $M$ be a manifold of constant curvature $-k^2$. Choose an origin $O$ in $M$, and let $x \in M$ be a point of distance $R_x$ from $O$. Then under Assumption \ref{assumptiondtotmoments} we have
\begin{equation}
\label{eq:incapp1}
        \mathbb{E}_x[\Delta R]=\frac{1}{k}\mathbb{E}_x[\log(\cosh k d_{\text{tot}}+\phi \sinh k d_{\text{tot}})]+O(R_x^{1-p})
\end{equation}
\begin{equation}
\label{eq:incapp2}
        \mathbb{E}_x[(\Delta R)^2]=\frac{1}{k^2}\mathbb{E}_x[\log^2(\cosh k d_{\text{tot}}+\phi \sinh k d_{\text{tot}})]+O(R_x^{2-p})
\end{equation}
where the implicit constants in the remainder terms depend only on $k$ and $B$.
\end{proposition}

\begin{proof}
Let $\alpha=\cosh{kR_x}$, $\beta=\sinh{kR_x}$, $c=\cosh{k d_{\text{tot}}}$ and $s=\sinh{k d_{\text{tot}}}$. 

Using Proposition \ref{prop:incrementexact}, followed by some algebraic manipulation, we find that
\small
\begin{align*}
    &\left| \Delta R-\frac{1}{k}\log{(c+\phi s)} \right| =\left| \frac{1}{k} \arccosh{(\alpha c + \phi \beta s)}-R_x-\frac{1}{k}\log{(c+\phi s)} \right| \\
    &= \left|\frac{1}{k} \log(2(\alpha c + \phi \beta s))  + \frac{1}{k} \log\left(\frac{1+\sqrt{1-(\alpha c + \phi \beta s)^{-2}}}{2}\right) - R_x - \frac{1}{k} \log(c+\phi s)\right| \\
    &= \left|\frac{1}{k} \log \left( \frac{2(\alpha c + \phi \beta s) \times e^{-kR_x}}{c+\phi s} \right) + \frac{1}{k}  \log\left(\frac{1+\sqrt{1-(\alpha c + \phi \beta s)^{-2}}}{2}\right)\right| \\
    &= \left|\frac{1}{k} \log \left( \frac{c(1+e^{-2kR_x}) + \phi s (1-e^{-2kR_x}) }{c+\phi s} \right) + \frac{1}{k}  \log\left(\frac{1+\sqrt{1-(\alpha c + \phi \beta s)^{-2}}}{2}\right)\right| \\
    &=  \left| \frac{1}{k} \log\left(1+\frac{c-\phi s}{c+\phi s} e^{-2kR_x}\right)+\frac{1}{k}\log\left(\frac{1+\sqrt{1-(\alpha c + \phi \beta s)^{-2}}}{2}\right)\right| \\
    &\leq  \left| \frac{1}{k} \log\left(1+\frac{c-\phi s}{c+\phi s} e^{-2kR_x}\right) \right| +\left| \frac{1}{k}\log\left(\frac{1+\sqrt{1-(\alpha c + \phi \beta s)^{-2}}}{2}\right) \right| \\
\end{align*}
Note that, since $\alpha, \beta, c, s \geq 0$ and $\phi \in [-1,1]$, we have
\[ e^{-kd_{\text{tot}}} = c-s \leq c \pm \phi s \leq c+s=e^{kd_{\text{tot}}} \: \: (\dagger)\]
In particular, $c \pm \phi s \geq 0$. Further,
\[ \alpha c + \phi \beta s \geq \alpha c - \beta s = \cosh(k R_x - k d_{\text{tot}}) \geq \max\left(\frac{1}{2} \exp(|kR_x - kd_{\text{tot}}|), \: 1 \right). \: \: (\ddagger) \] 

One can verify that if $u \geq 0$ then $|\log(1+u)| \leq u$ and that if $u \geq 1$ then 
\[
\left\lvert \log  \left( \frac{1+\sqrt{1-u^{-2}}}{2}  \right) \right\rvert \leq \frac{1}{u^2}.
\]
It follows that
\begin{align*}
    \left| \Delta R-\frac{1}{k}\log{(c+\phi s)} \right| &\leq \frac{1}{k} \frac{c-\phi s}{c+\phi s} e^{-2kR_x} + \frac{1}{k} \frac{1}{(\alpha c + \phi \beta s)^2} \\
    &\leq \frac{1}{k} e^{2k(d_{\text{tot}}-R_x)} + \frac{4}{k} e^{2k(d_{\text{tot}}-R_x)} \: \: \: \text{(by} \dagger \text{and} \ddagger\text{)}\\
    &= \frac{5}{k} e^{2k(d_{\text{tot}}-R_x)} \hspace{10pt} (*)
\end{align*}
Let $E$ be the event that $d_{\text{tot}} \leq R/2$, and $E^c$ its complement. Then
\begin{equation*}
\begin{split}
\left|\mathbb{E}_x\left[ \Delta R -\frac{1}{k}\log{(c+\phi s)} \right] \right| \leq \left|\mathbb{E}_x\left[ \left( \Delta R -\frac{1}{k}\log{(c+\phi s)} \right) 1_E \right] \right| + \\
\left|\mathbb{E}_x\left[ \left( \Delta R -\frac{1}{k}\log{(c+\phi s)} \right) 1_{E^c}\right] \right| \\
:= Q_1 + Q_2
\end{split}
\end{equation*}
By $(*)$,
\begin{align*}
Q_1 &\leq \mathbb{E}\left[  \frac{5}{k} e^{2k(d_{\text{tot}}-R_x)} 1_E \right] \\
&\leq \mathbb{E}\left[\frac{5}{k}e^{-kR_x} 1_E\right] \\
&\leq \frac{5}{k}e^{-kR_x}
\end{align*}
To bound $Q_2$, we use Assumption \ref{assumptiondtotmoments}:
\begin{align*}
    Q_2 \leq \mathbb{E}_x[2 d_{\text{tot}} 1_{E^c}] = \mathbb{E}_x[2 d_{\text{tot}}^p d_{\text{tot}}^{1-p} 1_{E^c}] &\leq \left(\frac{R_x}{2}\right)^{1-p} \mathbb{E}_x[2d_{tot}^p 1_{E^c}] \\
    &\leq 2^{p} R_x^{1-p} B
\end{align*}
Combining these bounds establishes \eqref{eq:incapp1}. For \eqref{eq:incapp2}, we note from the elementary observation that $a^2-b^2 = (a-b)^2 + 2b(a-b)$ that
{\small
\begin{equation*}
\begin{split}
\left|\mathbb{E}_x\left[ (\Delta R)^2 -\left(\frac{1}{k}\log{(c+\phi s)}\right)^2 \right] \right| \\ 
\leq \mathbb{E}_x\left[\left| \left( \Delta R -\frac{1}{k}\log{(c+\phi s)} \right)^2  1_E - \frac{2}{k} \log(c+\phi s) \left( \Delta R -\frac{1}{k}\log{(c+\phi s)} \right) 1_E \right| \right] \\
+\mathbb{E}_x\left[\left| \left( \Delta R -\frac{1}{k}\log{(c+\phi s)} \right)^2  1_{E^c} - \frac{2}{k} \log(c+\phi s) \left( \Delta R -\frac{1}{k}\log{(c+\phi s)} \right) 1_{E^c} \right| \right] \\
:= Q_3 + Q_4
\end{split}
\end{equation*}
}
Using $(*)$, we find that
\begin{align*}
    Q_3 &\leq \mathbb{E}_x\left[\left(\left(\frac{5}{k}e^{2k(d_{\text{tot}}-R_x)}\right)^2+ 2d_{\text{tot}} \times \left(\frac{5}{k}e^{2k(d_{\text{tot}}-R_x)}\right)\right) 1_E \right] \\
    &\leq \mathbb{E}_x\left[ \frac{25}{k^2} e^{-2kR_x} + \frac{5}{k} e^{-kR_x} d_{\text{tot}}\right]
\end{align*}
which is of the required form by Assumption \ref{assumptiondtotmoments}. Finally,
\begin{align*}
Q_4 &\leq \mathbb{E}_x[8 \: d_{\text{tot}}^2 1_{E^c} ] \\
       &=\mathbb{E}_x[8 \: d_{\text{tot}}^p \: d_{\text{tot}}^{2-p}  1_{E^c}] \\
       &\leq 2^{p+1}R^{2-p} \: \mathbb{E}_x[d_{\text{tot}}^{p} 1_{E^c}] \\  
       &\leq 2^{p+1} B \: R^{2-p}
\end{align*}
which gives \eqref{eq:incapp2}.

\end{proof}

\begin{proof}[Proof of Theorem \ref{thm:mainresultA}]
Proposition \ref{prop:incrementapproximation} tells us that there exists a constant $C>0$, independent of $x \in M$, such that, for $i=1,2$,
\begin{equation*}
   \mathbb{E}_x\left[\frac{1}{k^i} \log^i(c+\phi s)\right] - Cr^{i-p} \leq \mathbb{E}_x[(\Delta R)^i] \leq \mathbb{E}_x\left[\frac{1}{k^i}\log^i(c+\phi s)\right] + Cr^{i-p}
\end{equation*}
Taking suprema and infima over $S(r)$, we find that if we let
\begin{align*}
    \underline{\mu}_i(r) &= \underline{\nu}_i(r) - Cr^{i-p} \\
    \overline{\mu}_r(r) &= \overline{\nu}_i(r) + Cr^{i-p}
\end{align*}
then condition \eqref{eq:mucondition} is satisfied. It suffices, therefore, to check that if the assumptions in (i) (respectively (ii)) hold for $\nu$, then they also hold for $\mu$. For (i),
\begin{align*}
    \liminf_{r \rightarrow \infty}(2r \underline{\mu}_1 - \overline{\mu}_2) &= \liminf_{r \rightarrow \infty}(2r \underline{\nu}_1 - \overline{\nu}_2-3Cr^{2-p}) \\
    &\geq \liminf_{r \rightarrow \infty} (2r \underline{\nu}_1 - \overline{\nu}_2) + \liminf_{r \rightarrow \infty} 3Cr^{2-p} \\
    & \geq 0
\end{align*}
Assume that (ii) holds for $\nu$ for some constants $r_0 \geq 0$ and $\theta>0$, where $\theta=2\theta'$. Then
{ \small
\begin{align*}
    2r \overline{\mu}_1 &- \left(1+\frac{1-\theta'}{\log r} \right)\underline{\mu}_2 = 2r\overline{\nu}_1 - \left(1+\frac{1-\theta'}{\log r} \right) \underline{\nu}_2+C\left(3+\frac{1-\theta'}{\log r} \right)r^{2-p} \\
    &=\left(2r\overline{\nu}_1 -\left(1+\frac{1-\theta}{\log r}  \right) \underline{\nu}_2 \right) - \frac{\theta' \underline{\nu}_2}{\log r} + C\left(3+\frac{1-\theta'}{\log r} \right)r^{2-p}  \: \: (*)
\end{align*}
} 
\noindent By assumption, $\liminf_{r \rightarrow \infty} \underline{\nu_2}>0$, and therefore  $\frac{\theta' \underline{\nu}_2}{\log r}$ decays more slowly than $C\left(3+\frac{1-\theta'}{\log r} \right)r^{2-p}$. It follows that $(*)$ will be negative for all sufficiently large $r$, as required.
\end{proof}

\section{Generalising to Non-Constant Curvature}
\label{section:nonconstantcurvature}

Given a manifold $M$ with non-constant sectional curvature, we can sometimes reduce to the constant-curvature case using the following consequence of the Rauch comparison theorem.

\begin{theorem} \label{thm:rauch}
Let $M$ and $M'$ be complete and simply connected Riemannian manifolds of everywhere nonpositive sectional curvature. Suppose that for all points $p \in M, p' \in M'$ and planes $\Pi \subset T_pM$, $\Pi' \subset T_{p'}M'$, the sectional curvature satisfies $K(p',\Pi') \geq K(p,\Pi)$. Let $p \in M$, $p' \in M'$ and fix a linear isometry $i: T_pM \rightarrow T_{p'}M'$. Given a curve $c: [0,1] \rightarrow M$, define a corresponding curve $c': [0,1] \rightarrow M'$ by $c'(s)=\exp_{p'} \circ i \circ \exp^{-1}_p(c(s))$. Then $\text{Length}(c) \geq \text{Length}(c')$. 
\end{theorem}
\begin{proof}
Under these circumstances, the exponential maps $TM \rightarrow M$, $TM' \rightarrow M'$ are diffeomorphisms. The statement therefore follows from Chapter 10, Proposition 2.5 of \cite{carmo1992riemannian}. 
\end{proof}

\begin{proposition}
\label{prop:rauchconsequence}
Let $M$ be a complete simply connected manifold whose sectional curvature is negative and of magnitude at least $k^2>0$ at every point. Then, in the notation of Proposition \ref{prop:incrementapproximation},
\begin{equation*}
        \mathbb{E}_x[\Delta R] \geq 
        \frac{1}{k}\mathbb{E}_x[\log(\cosh k d_{\text{tot}}+\phi \sinh k d_{\text{tot}})]+O(R_x^{1-p})
\end{equation*}
Suppose, in addition, that the sectional curvature of $M$ has magnitude at most $K^2$ at every point. Then
\begin{equation*}
\begin{split}
        \mathbb{E}_x[(\Delta R)^2] \geq 
        \frac{1}{k^2}\mathbb{E}_x[\log^2(\cosh k d_{\text{tot}}+\phi \sinh k d_{\text{tot}})1_{(\Delta R) \geq 0)}] \: + \\
        \frac{1}{K^2}\mathbb{E}_x[\log^2(\cosh K d_{\text{tot}}+\phi \sinh K d_{\text{tot}})1_{(\Delta R) < 0)}]  
\end{split}
\end{equation*}
\end{proposition}
\begin{proof}
Let $R$, $d_{\text{tot}}$ and $d_{\text{rad}}$ be fixed, and let $M'$ be a manifold of constant curvature $-k^2$. Let $O,p_n, p_{n+1} \in M$ and $O',p_n',p_{n+1}'$ such that $\text{dist}_M(O,p_n)=\text{dist}_{M'}(O',p_n')=R$, and such that $p_{n+1}, p_{n+1}'$ are both consistent with the specified values of $d_{\text{tot}}$ and $d_{\text{rad}}$. Then we claim that $\text{dist}_M(O,p_{n+1}) \geq \text{dist}_{M'}(O',p_{n+1}')$. Indeed, consider triangle $Op_{n}p_{n+1}$. By applying $\exp^{-1}_{p_n}$ we obtain a triangle $O''p_{n}''p_{n+1}'' \in T_{p_n} M$ such that $O''p_{n}''=OP$ and $p_{n}p_{n+1}=p''_{n}p_{n+1}''=d_{\text{tot}}$. Similarly, applying $\exp^{-1}_{p'}$ gives a triangle $O'''p_{n}'''p_{n+1}'''$. The exponential map preserves distances from $p_{n}''$ (or $p_{n}'''$), and so $O''p_{n}''=O'''p_{n}'''$ and $p_{n}''p_{n+1}''=p_{n}'''p_{n+1}'''=d_{\text{tot}}$. Also, because $d_{\text{rad}}$ is fixed, the angles $O''p_{n}''p_{n+1}''$ and $O'''p_{n}'''p_{n+1}'''$ are equal. It follows that there exists an isometry $i$ mapping $O''p_{n}''p_{n+1}''$ to $O'''p_{n}'''p_{n+1}'''$. Applying Theorem \ref{thm:rauch}, our claim follows. Now choose  $M'$ to be the manifold of constant curvature $-k^2$ and then the first part of the proposition follows from taking expectations, together with Proposition \ref{prop:incrementapproximation}. \\

For the second part, we may repeat the argument in the first part to obtain both a lower and an upper bound for $(\Delta R)$. Recall that for real numbers $x$ and $y$ such that $x \leq y$, we have $x^2 \leq y^2$ if $x \geq 0$ whilst $x^2 \geq y^2$ if $y \leq 0$. Therefore if we write $(\Delta R)^2 = (\Delta R)^2 1_{\Delta R \geq 0} + (\Delta R)^2 1_{\Delta R<0}$ we deduce a lower bound for both these terms, and hence obtain the result.
\end{proof}

These estimates may be used in conjunction with Theorem \ref{thm:mainresultA} to obtain recurrence-transience criteria for more general manifolds. For example

\begin{theorem}
\label{thm:mainresultB}
Let $M$ be a complete simply connected manifold whose sectional curvature is negative and of magnitude at least $k^2>0$ at every point. Let $X$ be a Markov chain on $M$. \\
(i) Suppose that
\begin{equation*}
    \lim_{r \rightarrow \infty} r \inf_{x \in S(r)}\mathbb{E}_x[\log(\cosh(kd_{\text{tot}})+ \phi \sinh(kd_{\text{tot}}))] = \infty
\end{equation*}
then the chain is transient. \\
(ii) Suppose instead  that the sectional curvature $\Pi$ satisfies $-K^2 \leq \Pi \leq -k^2$ at each point, that
\begin{equation*}
\begin{split}
        \lim_{r \rightarrow \infty} \inf_{x \in S(r)} \frac{1}{k}\mathbb{E}_x[\log^2(\cosh k d_{\text{tot}}+\phi \sinh k d_{\text{tot}})1_{(\Delta R) \geq 0)}] \: + \\
        \frac{1}{K}\mathbb{E}_x[\log^2(\cosh K d_{\text{tot}}+\phi \sinh K d_{\text{tot}})1_{(\Delta R) < 0)}] > 0  
\end{split}
\end{equation*}
and that
\begin{equation*}
    \lim_{r \rightarrow \infty} r \inf_{x \in S(r)}\mathbb{E}_x[\log(\cosh(kd_{\text{tot}})+ \frac{d_{\text{rad}}}{d_{\text{tot}}} \sinh(kd_{\text{tot}}))] = 0
\end{equation*}
then the chain is recurrent.
\end{theorem}

\section{A Criterion for Non-Confinement}
\label{section:nonconfinement}

Here we discuss sufficient conditions to have $\limsup{R_n}=\infty$ almost surely. As stated in Section \ref{section:assumptions}, in applications this is usually straightforward to prove. One such condition is Equation (3.10) of  \cite{menshikov2016non} which says that the radial process exits any interval $[0,r]$ in a finite number of steps with positive probability, depending only on $r$.

\begin{proposition}
\label{prop:localescape}
Suppose that for each $r \in \mathbb{R}_{\geq 0}$, there exist $a_r \in \mathbb{N}$ and $\delta_r>0$ such that, for all $n \geq 0$,
\[ \mathbb{P}[ \max_{n \leq m \leq n + a_r} R_m \geq r \mid \mathcal{F}_n] \geq \delta_r \text{ on } \{ R_n \leq r \} \]
Then Assumption \ref{assumptionnotconfined} holds.
\end{proposition}

We state another condition which applies to the class of chains satisfying $\mathbb{E}_x[d_{\text{rad}}]=0$ for all $x \in M$, and therefore to all zero-drift chains.
\begin{proposition}
\label{prop:radialnondegeneracycriterion}
Let $M$ be a manifold satisfying Assumption \ref{assumptionnegativecurvature}. Let $X$ be a Markov chain on $M$ satisfying Assumption \ref{assumptiondtotmoments}. Suppose that $\mathbb{E}_x[d_{rad}]=0$ and that there exists $\epsilon>0$ such that 
$\mathbb{E}_x[d_{rad}^2] \geq \epsilon$ for all $x \in M$. Then Assumption \ref{assumptionnotconfined} holds.
\end{proposition}

The rest of this section is devoted to the proof of this result.

\begin{proposition}
\label{prop:radialsubmartingale}
Let $X_1, X_2, X_3, \dots$ be a Markov chain on $M$, where $M$ satisfies Assumption $\ref{assumptionnegativecurvature}$. Assume that $\mathbb{E}_x[d_{\text{rad}}] \geq 0$ for all $x \in M$. Then the radial process $R_n:=\text{Dist}_M(O,X_n)$ is a nonnegative submartingale.
\end{proposition}
\begin{proof}
By Theorem \ref{thm:rauch}, it is enough to prove this in the Euclidean case. Using the cosine rule for triangles in $\mathbb{R}^2$, one can show that
\begin{equation*}
    \text{Dist}_M(O,X_{n+1}) = R_n \sqrt{1+\frac{2 d^n_{\text{rad}}}{R_n}+\frac{(d^n_{\text{tot}})^2}{R_n^2}}
\end{equation*}
Since $d^n_{\text{tot}} \geq |d^n_{\text{rad}}|$, we deduce that
\[ R_{n+1}-R_{n} \geq d_{\text{rad}} \]
for all values of $R_n, d_{\text{rad}}, d_{\text{tot}}$, and the result follows upon taking expectations.
\end{proof}

Let $(Y_n)_{n \in \mathbb{N}}$ be a process in $\mathbb{R}^d$ adapted to $(\mathcal{F}_n)_{n \in \mathbb{N}}$, and suppose that $Y_0=0$. Recall (or see \cite{williams1991probability}) that $Y$ has a Doob decomposition
\begin{equation}
    \label{eq:Doobdecomposition}
     Y_n = L_n + A_n
\end{equation}
where $L$ is an $\mathcal{F}$-martingale and $A$ is a predictable process, given by
\[ A_n = \sum_{i=0}^{n-1} \mathbb{E}[X_{i+1}-X_{i} \mid \mathcal{F}_{i}] \]
If $Y$ is a submartingale then $A$ is nonnegative and increasing.

\begin{lemma}
\label{lemma:martingalezn}
Let $(Y_n)_{n \in \mathbb{N}}$ be a process in $\mathbb{R}^d$ adapted to $(\mathcal{F}_n)_{n \in \mathbb{N}}$. Let $(L_n)_{n \in \mathbb{N}}$ and $(A_n)_{n \in \mathbb{N}}$ be as in Equation \eqref{eq:Doobdecomposition}. Assume that $\mathbb{E}[|\Delta L_n|^p \mid \mathcal{F}_n] \leq B$ and $\mathbb{E}[|\Delta A_n|^p \mid \mathcal{F}_n] \leq B$ almost surely for some $p>2, B \in \mathbb{R}$. After enlarging $\Omega$ if necessary, consider the process $(Z_n)_{n \in \mathbb{N}}$ given by $Z_0=Y_0$ and 
\[ (\Delta Z_n)=(\Delta L_n) + \xi_n(\Delta A_n) \]
where the $\xi_n$ are equal to $\pm 1$ with equal probability, independently of each other, $L$ or $A$. Then
\begin{enumerate}[(i)]
\item $\mathbb{E}[|\Delta Z_n|^p \mid \mathcal{F}_n] \leq B'$ for some $B'$ depending only on $B,p$, and $d$,
\item $Z$ is an $\mathcal{F}$-martingale,
\item $\mathbb{E}[|\Delta Y_n|^2 \mid \mathcal{F}_n]=\mathbb{E}[|\Delta Z_n|^2 \mid \mathcal{F}_n].$
\end{enumerate}
\end{lemma}
\begin{proof}
\begin{enumerate}[(i)]
    \item This follows from the bounds on $\mathbb{E}[|\Delta L|^p]$, $\mathbb{E}[|\Delta A|^p]$, and the inequality $|x+y|^p \leq C_{d,p}(|x|^p + |y|^p)$
for vectors $x,y \in \mathbb{R}^d$, where $C$ is a constant.
\item We check that
\begin{equation*}
\mathbb{E}[\Delta Z_n \mid \mathcal{F}_n]= \mathbb{E}[\Delta L_n \mid \mathcal{F}_n] + (\Delta A_n) \mathbb{E}[ \xi_n \mid \mathcal{F}_n] = 0+0 = 0
\end{equation*}
And that by part (i) and Lyapunov's inequality, there is a constant $B''$ such that
\begin{equation*}
\mathbb{E}[|Z_n|] \leq \sum_{k=1}^n \mathbb{E}[|\Delta Z_n] \leq n B'' < \infty
\end{equation*}
\item We calculate \begin{align*}
    \mathbb{E}[|\Delta Z_n|^2 \mid \mathcal{F}_n]&=\mathbb{E}\left[ (\Delta L_n^i + \xi_n \Delta A_n^i)^2 \mid \mathcal{F}_n\right] \\
    &= \mathbb{E}\left[ (\Delta L^i_n)^2 + 2 \xi_n \Delta A_n^i \Delta L_n^i + (\Delta A^i_n)^2 \mid \mathcal{F}_n \right] \\
    &= \mathbb{E}[|\Delta L_n|^2 + |\Delta A_n|^2 \mid \mathcal{F}_n]
\end{align*}
and similarly $\mathbb{E}[(\Delta Y_n)^2 \mid \mathcal{F}_n]=\mathbb{E}[(\Delta L_n)^2 + (\Delta A_n)^2 \mid \mathcal{F}_n]$.
\end{enumerate}
\end{proof}

\begin{proposition}
\label{prop:unboundednessofyn}
Let $(Y_n)_{n \in \mathbb{N}}$ be an $\mathbb{R}$-valued submartingale. Let $(L_n)_{n \in \mathbb{N}}$ and $(A_n)_{n \in \mathbb{N}}$ be as in Equation \eqref{eq:Doobdecomposition}. Assume that
\begin{enumerate}[(i)]
    \item There exists $B \in \mathbb{R}$, $p>2$ such that $\mathbb{E}[|\Delta L_n|^p \mid \mathcal{F}_n] \leq B$ and $\mathbb{E}[|\Delta A_n|^p \mid \mathcal{F}_n] \leq B$ for all $n$.
    \item There exists $\epsilon>0$ such that $\mathbb{E}[|\Delta Y_n|^2 \mid \mathcal{F}_n] \geq \epsilon$ for all $n$
    \item $\mathbb{P}[\lim_{n \rightarrow \infty}{L_n} = -\infty]=0$
\end{enumerate}
Then $\mathbb{P}[\limsup_{n \rightarrow \infty}{|Y_n| = \infty}]=1$.
In particular, if $Y$ is bounded below a.s. then 
$\mathbb{P}[\limsup_{n \rightarrow \infty}{Y_n = \infty}]=1$.
\end{proposition}
\begin{proof}
Let $Z=(Z_n)_{n \in \mathbb{N}}$ be as in Lemma \ref{lemma:martingalezn}. We claim that
\begin{equation*}
\begin{split} 
\{ \omega \in \Omega: \limsup_{n \rightarrow \infty}{|Y_n| = \infty} \} \supseteq \{ \omega \in \Omega: \limsup_{n \rightarrow \infty} {|Z_n|} = \infty \} \\ \cap \: \{ \omega \in \Omega: \lim_{n \rightarrow \infty}{L_n} \neq -\infty \}. \: \: \: (*)
\end{split}
\end{equation*}
To see this, first suppose that $\omega \in \Omega$ is such that $A$ is bounded, say $A_n \leq C$ for all $n$. Then $Y_n \geq Z_n - 2 C$ for all $n$. On the other hand, if $A$ is unbounded, then, since $A$ is positive and increasing, $\lim_{n \rightarrow \infty} A_n = \infty$ and so $\limsup_{n \rightarrow \infty} A_n + L_n$ will be infinity provided that there exists $C \in \mathbb{R}$ such that $L_n > C$ infinitely often. This establishes $(*)$. By (iii), it therefore suffices to prove that $\limsup |Z_n| \rightarrow \infty$ almost surely. Lemma \ref{lemma:martingalezn}, combined with Proposition 2.1 in \cite{georgiou2016anomalous}, establishes this result.
\end{proof}

\begin{proof}[Proof of Proposition \ref{prop:radialnondegeneracycriterion}:]
Decompose the radial process as $R_n=L_n+A_n$. It follows from Proposition \ref{prop:radialsubmartingale} and the uniqueness of the Doob decomposition that $L_n = \sum_{i=1}^n d_{\text{rad}}^{i}$. It suffices to check that the assumptions of Proposition \ref{prop:unboundednessofyn} hold when $Y_n=R_n$.
\begin{enumerate}[(i)]
    \item Almost surely, $|\Delta L_n| = |d_{\text{rad}}^n| \leq d^n_{\text{tot}}$ and $|\Delta L_n+ \Delta A_n| \leq d_{\text{tot}}$ and hence $|\Delta A_n| \leq 2d_{\text{tot}}$. Now apply Assumption \ref{assumptiondtotmoments}.
    \item $\mathbb{E}[|\Delta R_n|^2 \mid \mathcal{F}_n] \geq \mathbb{E}[|\Delta L_n|^2 \mid \mathcal{F}_n] \geq \epsilon$
    \item Theorem 2.2 in \cite{georgiou2016anomalous} shows that almost surely there is a bounded neighbourhood of the origin $N$ such that $L_n \in N$ infinitely often.
\end{enumerate}
\end{proof}

In the case $M=\mathbb{R}^n$, our sufficient condition for non-confinement may be compared to the corresponding result in \cite{georgiou2016anomalous} (Proposition 2.1).  Our result applies to a wider class of processes (not just martingales), but at the price of being a little more restrictive - we require $\mathbb{E}[d_{\text{rad}}^2 \geq \epsilon]$ as opposed to $\mathbb{E}[d_{\text{tot}}^2 \geq \epsilon]$. For applications of these criteria to examples, see Section \ref{section:examples}.

\section{Comparison of the Euclidean and Hyperbolic Cases}
\label{section:eucvshyp}

We briefly recall some results about the Euclidean case, with $X$ a Markov chain on $M=\mathbb{R}^d$ satisfying Assumptions \ref{assumptiondtotmoments} and \ref{assumptionnotconfined}. If there exists $\epsilon>0$ such that $\mathbb{E}_x[d_{\text{rad}}] \geq \epsilon$ whenever $\text{Dist}_{\mathbb{R}^d}(x,O)$ is sufficiently large, then $X$ will be transient, whilst if $\mathbb{E}_x[d_{\text{rad}}] \leq -\epsilon$ whenever  $\text{Dist}_{\mathbb{R}^d}(x,O)$  is sufficiently large, then $X$ will be recurrent. In the case $\mathbb{E}_x[d_{\text{rad}}]=0$ for all $x \in M$, \cite{georgiou2016anomalous} gives a criterion in terms of the second moments; if $\mathbb{E}_x[d_{\text{rad}}^2] \rightarrow U$ and $\mathbb{E}_x[d_{\text{tot}}^2] \rightarrow V$ as $|x| \rightarrow \infty$ then $X$ is recurrent if $2U > V$ and transient if $2U<V$. (The boundary case $2U=V$ is also considered in \cite{georgiou2016anomalous} but we do not discuss it here).

If $M$ has constant curvature $-k^2$ then the important quantity is 
\[ 
F(k,d_{\text{rad}},d_{\text{tot}}) := \frac{1}{k} \log\left(\cosh(k d_{\text{tot}}) + \frac{d_{\text{rad}}}{d_{\text{tot}}} \sinh(k d_{\text{tot}})\right)
\]
which takes values on the domain  $\mathcal{D}:=\{ (k,d_{\text{rad}},d_{\text{tot}}) \in  (0, \infty) \times [0,\infty) \times (-\infty, \infty) : d_{\text{tot}} \geq |d_{\text{rad}}| \}$. 

\begin{proposition}
Let $M$ satisfy Assumption \ref{assumptionnegativecurvature} with sectional curvature at most $-k^2$ at every point for some $k>0$. Let $X$ be a Markov chain on $M$ satisfying Assumptions
\ref{assumptiondtotmoments} and \ref{assumptionnotconfined}. \\
(i) If there exists $\epsilon>0$ such that $\mathbb{E}_x[d_{\text{rad}}] \geq \epsilon$ whenever $\text{Dist}_{\mathbb{R}^d}(x,O)$ is sufficiently large, then $X$ is $O$-transient. \\
(ii) For any $N>0$ there exists $O$-transient $X$ such that $\mathbb{E}_x[d_{\text{rad}}] \leq -N$ whenever $\text{Dist}_{\mathbb{R}^d}(x,O)$ sufficiently large. 
\end{proposition}
\begin{proof}
(i) This follows easily from the inequality $\mathbb{E}_x[\Delta R] \geq \mathbb{E}_x[d_{\text{rad}}]$ in Euclidean space, together with the comparison theorem. \\
(ii) For each $x \in M$ with  $\text{Dist}_{\mathbb{R}^d}(x,O)$, take the distributions of $d_{\text{tot}}$ and $d_{\text{rad}}$, conditional on the walk currently being at $x \in M$, to be
\begin{equation*}
    d_{\text{tot}}=4N \text{ almost surely, and } d_{\text{rad}} = \begin{cases} -2N & \text{ with probability } 1/2 \\
    0 & \text{ with probability } 1/2
    \end{cases}
\end{equation*}
then $\mathbb{E}_x[d_{\text{rad}}] = -N$. One can check that, for all fixed $k$, the function
$\mathbb{E}_x[F(k,d_{\text{rad}},d_{\text{tot}})]=\frac{F(k,0,4N)+F(k,-2N,4N)}{2} \rightarrow \infty$
as $N \rightarrow \infty$ and hence, by increasing the value of $N$ if necessary, we may assume that $\mathbb{E}_x[F] \geq 1$ whenever $\text{Dist}_{\mathbb{R}^d}(x,O)$ is sufficiently large. Theorem \ref{thm:mainresultB} then gives transience.
\end{proof}

We now consider the case $\mathbb{E}_x[d_{\text{rad}}]=0$ whenever  $\text{Dist}_{\mathbb{R}^d}(x,O)$ is large enough. This case is important because it contains all zero drift chains on $M$.  It is not possible, for any fixed $k$, to approximate $F(k,\cdot,\cdot)$ uniformly by a bivariate polynomial in $d_{\text{rad}}, d_{\text{tot}}$, so we do not expect any finite collection of moments to provide complete information about $F$. However, we have the following estimate.

\begin{lemma}
\label{lemma:fgrowthestimate}
For all $(k,d_{\text{rad}},d_{\text{tot}}) \in \mathcal{D}$,
\begin{align*}
  d_{\text{rad}} + J_{min}(k,d_{\text{tot}})(d_{\text{tot}}^2-d_{\text{rad}}^2) &\leq F(k,d_{\text{rad}},d_{\text{tot}}) \\
  &\leq d_{\text{rad}} + J_{max}(k,d_{\text{tot}})(d_{\text{tot}}^2-d_{\text{rad}}^2)  
\end{align*}
where
\begin{align*}
J_{min}(k,d_{\text{tot}}) &=  \frac{1}{2d_{\text{tot}}^2}\left( d_{\text{tot}} - \frac{\sinh(k d_{\text{tot}})}{k(\cosh(k d_{\text{tot}})+\sinh(k d_{\text{tot}}))} \right) \\
J_{max}(k,d_{\text{tot}}) &= \frac{1}{2d_{\text{tot}}^2}\left( -d_{\text{tot}} + \frac{\sinh(k d_{\text{tot}})}{k(\cosh(k d_{\text{tot}})-\sinh(k d_{\text{tot}}))} \right)
\end{align*}
Moreover, $J_{min}$ is positive, increasing in $k$ and decreasing in $d_{\text{tot}}$, whilst $J_{max}$ is nonnegative and increasing in both $k$ and $d_{\text{tot}}$.
\end{lemma}
\begin{proof}
As usual, let $\phi = d_{\text{rad}}/d_{\text{tot}}$, so that $\phi \in [-1,1]$. For fixed $k,d_{\text{tot}}$, consider the function
\[ G(\phi) := \left(\frac{1}{k} \log(\cosh(kd_{\text{tot}})+\phi \sinh(k d_{\text{tot}}))-\phi d_{\text{tot}} \right)(1-\phi^2)^{-1} \]
It is lengthy but elementary to check that $G$ is decreasing on $\phi \in [-1,1]$ and that its limits at $\phi=\pm 1$ are
\[ \frac{1}{2}\left( \pm d_{\text{tot}} \mp \frac{\sinh(k d_{\text{tot}})}{k(\cosh(kd_{\text{tot}}) \pm \sinh(kd_{\text{tot}}))}\right)
\]
the first part follows, and the remainder follows from a direct check.
\end{proof}

\begin{theorem}
\label{thm:ellipticimpliestransient}
Fix $k>0$. Let $X$ be a zero drift Markov chain on a manifold whose sectional curvature is at most $-k^2$ at every point. Suppose that Assumptions \ref{assumptionnegativecurvature}, \ref{assumptiondtotmoments} and \ref{assumptionnotconfined} are satisfied. Suppose also that there exist constants $D_{\text{min}}$ and $\epsilon>0$ such that 
\begin{equation}
\label{eq:sufficientforuniformlyelliptic}
    \mathbb{E}_x[d_{\text{tot}}^2 - d_{\text{rad}}^2] \geq \epsilon
\end{equation}
for every $x \in M$ such that $\text{Dist}_M(O,x) \geq D_{min}$. Then the chain is transient.
\end{theorem}
\begin{proof}
Let $c=\cosh(k d_{\text{tot}}), s=\sinh(k d_{\text{tot}})$ and $\phi=d_{\text{rad}}/d_{\text{tot}}$. Let $A$ be a constant to be chosen later. Then
\begin{align*}
    \frac{1}{k} \mathbb{E}_x[\log(c+\phi s)] &\geq \mathbb{E}_x[d_{\text{rad}}] +    \frac{1}{2} \mathbb{E}_x\left[\frac{1}{d_{\text{tot}}^2}\left( d_{\text{tot}}-\frac{s}{k(c+s)} \right)(d_{\text{tot}}^2-d_{\text{rad}}^2) \right] \\
    &\geq \frac{1}{2} \mathbb{E}_x\left[\frac{1}{d_{\text{tot}}^2}\left( d_{\text{tot}}-\frac{s}{k(c+s)} \right)(d_{\text{tot}}^2-d_{\text{rad}}^2)1_{d_{\text{tot}}<A} \right] \\
    &= \frac{1}{2} \mathbb{E}_x\left[\left( \frac{1}{d_{\text{tot}}} - \frac{1}{2 k d_{\text{tot}}^2}(1-e^{-2kd_{\text{tot}}}) \right)(d_{\text{tot}}^2-d_{\text{rad}}^2)1_{d_{\text{tot}}<A} \right] \\
    &\geq \frac{1}{2} \mathbb{E}_x\left[\left( \frac{1}{A} - \frac{1}{2 k A^2}(1-e^{-2kA}) \right)(d_{\text{tot}}^2-d_{\text{rad}}^2)1_{d_{\text{tot}}<A} \right],
\end{align*}
where the last line follows from the fact that $\frac{1}{d_{\text{tot}}}-\frac{1}{2 k d_{\text{tot}}^2}(1-e^{-2kd_{\text{tot}}})$ is positive and decreasing in $d_{\text{tot}}$. Therefore there is a constant $A_0$, depending only on $k$, such that if $A>A_0$ and $\text{Dist}_M(O,x) \geq D_{\text{min}}$, then
\begin{align*}
    \frac{1}{k} \mathbb{E}_x[\log(c+\phi s)] &\geq  
    \frac{1}{4A} \mathbb{E}_x\left[(d_{\text{tot}}^2-d_{\text{rad}}^2)1_{d_{\text{tot}}<A} \right] \\
    &= \frac{1}{4A} \left( \mathbb{E}_x[d_{\text{tot}}^2-d_{\text{rad}}^2] - \mathbb{E}_x[(d_{\text{tot}}^2-d_{\text{rad}}^2) 1_{d_{\text{tot}}>A}] \right) \\
     &\geq \frac{1}{4A} \left( \epsilon - \mathbb{E}_x[d_{\text{tot}}^2 1_{d_{\text{tot}}>A}] \right)
\end{align*}

If $\alpha=\frac{p}{2}$ and $\beta=\frac{p}{p-2}$, then $\frac{1}{\alpha}+\frac{1}{\beta}=1$. Applying first H\"{o}lder's inequality, then Assumption \ref{assumptiondtotmoments}, then Markov's inequality and finally Lyapunov's inequality, we obtain
\begin{align*}
    \mathbb{E}_x[d_{\text{tot}}^2 1_{d_{\text{tot}}>A}] &\leq \mathbb{E}_x[d_{\text{tot}}^p]^{\frac{1}{\alpha}} \mathbb{P}_x[d_{\text{tot}} > A]^{\frac{1}{\beta}}  \\
    &\leq B \: \mathbb{P}_x[d_{\text{tot}} > A]^{\frac{1}{\beta}} \\
    &\leq B \left( \frac{\mathbb{E}_x[d_{\text{tot}}]}{A} \right)^{\frac{1}{\beta}} \\
    &\leq B^{1+\frac{1}{p\beta}} A^{-\frac{1}{\beta}}
\end{align*}
Choose $A$ sufficiently large that $B^{1+\frac{1}{p\beta}} A^{-\frac{1}{\beta}} \leq \frac{\epsilon}{2}$. This then gives
\begin{equation*}
    \frac{1}{k} \mathbb{E}_x[\log(c+\phi s)] \geq \frac{\epsilon}{2}
\end{equation*}
which implies transience by Theorem \ref{thm:mainresultB}. 
\end{proof}

\begin{proof}[Proof of Theorem \ref{corol:uniformlyelliptic}]
We need only check that If $X$ is uniformly elliptic, then
\eqref{eq:sufficientforuniformlyelliptic} automatically holds. To see this
choose an orthonormal basis for $T_p M$ of the form $\mathcal{B}_p = \{ \textbf{e}_1 = \textbf{e}_{rad}, \textbf{e}_2, \dots, \textbf{e}_d \}$ for each $p \in M$. If $X_n=x$ and $v = \exp^{-1}_{X_n}(X_{n+1})$, then
$v$ may be written uniquely in the form $v=\sum \lambda_i \textbf{e}_i$. If $X$ is uniformly elliptic then $\mathbb{P}[|\lambda_i| \geq \epsilon] \geq \epsilon$ and hence $\mathbb{E}[\lambda_i^2] \geq \epsilon^3$ for each $i$. But \[ \mathbb{E}_x[d_{\text{tot}}^2-d_{\text{rad}}^2] = \sum_{i=2}^d \lambda_i^2 \geq (d-1)\epsilon^3 \]
This completes the proof.
\end{proof}

\noindent In Euclidean space, under our assumptions, if a zero-drift chain satisfies
\begin{equation*}
   \lim_{r \rightarrow \infty} \: r \sup_{x \in S(r)} \mathbb{E}_x[ d_{\text{tot}}^2-d_{\text{rad}}^2]  = 0
\end{equation*}
then it is recurrent. We now show that this fails in hyperbolic space, for any polynomial growth factor.

\begin{proposition}
\label{prop:stillnotrecurrent}
There is a zero-drift transient chain in the hyperbolic plane such that, for every positive integer $N$,
\begin{equation*}
   \lim_{r \rightarrow \infty} \: r^N \sup_{x \in S(r)} \mathbb{E}_x[ d_{\text{tot}}^2-d_{\text{rad}}^2]  = 0
\end{equation*}
\end{proposition}
\begin{proof}
We give an example of such a chain. Take the probability density of $d_{\text{tot}}$, conditional on the chain being at $x \in M$, to be the same for every $x$ and given by
\begin{equation*}
     f_{tot}(y \mid x) = \frac{m-1}{y^m}; \hspace{5pt} 1 \leq y < \infty
\end{equation*}
where $m$ is a constant; it is necessary to choose $m>3$ in order for Assumption \ref{assumptiondtotmoments} to hold. For some function $\lambda(r)$ to be chosen later, let
\[ \epsilon(y)= \frac{1-\cosh(y)+\sinh(y)}{\sinh(y)} \cdot 1_{y \geq \lambda(r)}\]
and, conditional on $d_{\text{tot}}=y$ let $\phi:=\frac{d_{\text{rad}}}{d_{\text{tot}}}$ be distributed as
\[
\phi(\cdot \mid d_{\text{tot}}=y)=\begin{cases}
1 & \text{ with probability } \alpha(y) \\
-1+\epsilon(y) & \text{ with probability } 1-\alpha(y) \\
\end{cases}
\]

where
\[ \alpha(y) =  \frac{1-\epsilon(y)}{2-\epsilon(y)} = \begin{cases} \frac{1-\cosh(y)+\sinh(y)}{2} & \text{ if } y \geq \lambda(r) \\
\frac{1}{2} & \text{ otherwise } \end{cases} \]
Notice that $\epsilon$ and $\alpha$ depend upon the point $x \in M$ via its distance from $O$, although for brevity we omit this from our notation. One can check that $0 \leq \alpha(y) \leq 1$ for all $y \geq 1$, so that this definition makes sense. The choice of $\alpha$ ensures that $\mathbb{E}[\phi \mid d_{\text{tot}}=y]=0$ for all $y$, and hence that $\mathbb{E}_x[d_{\text{rad}}]=\mathbb{E}_x[\phi d_{tot}]=0$ for all $x \in M$. The choice of $\epsilon$ is made to simplify some of the forthcoming expectation calculations. Having specified the distributions of $d_{\text{rad}}$ and $d_{\text{tot}}$ it is straightforward to choose the transverse components to give a zero drift chain. We compute
\begin{align*}
     \mathbb{E}_x[d^2_{\text{tot}}(1-\phi^2) \mid d_{\text{tot}}=y] &= y^2 \: \mathbb{E}_x[1-\phi^2 \mid d_{\text{tot}}=y ] \\
     &= y^2\left(0+(1-\alpha)\left(1-(-1+\epsilon)^2\right) \right) \\
     &= y^2 \epsilon 
\end{align*}
From now on assume $\lambda(r) \geq 1$ for all $r$. Then
\begin{align*}
     \mathbb{E}_x[d_{\text{tot}}^2(1-\phi^2)] &= \int_{y=1}^\infty f(y) \: \mathbb{E}_x[d_{\text{tot}}^2(1-\phi^2) \mid d_{\text{tot}}=y] \: dy  \\
     &= \int_{y=\lambda}^\infty \frac{(m-1)}{\sinh(y)y^{m-2}}  (1-\cosh(y)+\sinh(y)) \: dy \\
     &\leq \int_{y=\lambda}^\infty \frac{2m}{e^y y^{m-2}} \: dy \\
     &\leq \int_{y=\lambda}^\infty 2m e^{-y} \: dy \\
     &= 2m e^{-\lambda} \: \: (*)
\end{align*}
On the other hand, letting $c=\cosh(y)$ and $s=\sinh(y)$, we find that, if $y \geq \lambda(r)$, then
\begin{align*}
     \mathbb{E}_x[\log(\cosh(d_{\text{tot}}) &+\phi \sinh(d_{\text{tot}})) \mid d_{\text{tot}}=y] = \mathbb{E}_x[\log(c+\phi s)] \\
     &=\alpha \log(c+s) + (1-\alpha) \times \log(1) \\
     &= y \alpha(y) 
\end{align*}
whilst if $y < \lambda(r)$ then $\mathbb{E}[\log(\cosh(d_{\text{tot}}) +\phi \sinh(d_{\text{tot}})) \mid d_{\text{tot}}=y] = 0$. So 
\begin{align*}
     \mathbb{E}_x[\log(\cosh(d_{\text{tot}})+\phi \sinh(d_{\text{tot}}))] &= \int_{y=\lambda}^\infty  y \: \alpha(y) \: f(y) \: dy \\
     &= \int_{y=\lambda}^\infty \frac{(m-1)(1+\sinh(y)-\cosh(y))}{2y^{m-1}} \: dy \\
     &\geq \int_{y=\lambda}^\infty \frac{m-1}{4y^{m-1}} \: dy \\
     &= \frac{m-1}{4(m-2)} \frac{1}{\lambda^{m-2}} \: \: (**)
\end{align*}
Choose $\lambda(r)=r^{\frac{1}{m-1}}$. Then $(*)$ tells us that $\sup_{x \in S(r)} \mathbb{E}_x[d^2_{\text{tot}}-d^2_{\text{rad}}]$ has the required rate of decay, and $(**)$, together with Theorem \ref{thm:mainresultA}, tells us that the chain is transient.
\end{proof}

\section{Examples}
\label{section:examples}

In this section, we generalise the `Elliptic Random Walk Model' found in Section 3 of \cite{georgiou2016anomalous} to radially symmetric manifolds of negative curvature. Let $V$ be a finite-dimensional inner product space of dimension $d$. Given $v \in V$ and $a,b>0$, define $L_V(a,b,v): V \rightarrow V$ to be the linear transformation that sends $v$ to $av \sqrt{d} $ and any  $w \in \langle v \rangle^\perp$ to $b w \sqrt{d}$. Define an elliptical measure
\[ \xi_V(a,b,v): \text{Borel}(V) \rightarrow \mathbb{R}_{\geq 0}  \]
by $\xi_V = \mu_V \circ L_V^{-1}(a,b,v)$, where $\mu_V$ is the uniform measure on the unit sphere in $V$. Thus $\xi_V$ is supported on an ellipsoid whose principal axes have lengths $a\sqrt{d},b\sqrt{d},b\sqrt{d},\dots,b\sqrt{d}$. Given a $d$-dimensional manifold $M$ with origin $O \in M$, and functions $a, b: M \rightarrow \mathbb{R}_{\geq 0}$, define a measure $\mu_p: \text{Borel}(M) \rightarrow \mathbb{R}$ at each point $p \in M$ by
\[ \mu_p = \xi_{T_p M}\left(a(p),b(p),e_{\text{rad}} \right) \circ exp_p^{-1} \]
where, if $p=O$, we temporarily define $e_{\text{rad}}(O)$ to be some fixed unit-length vector in $T_O M$ (as far as recurrence and transience is concerned, this choice is unimportant). This defines what we shall refer to as the \textit{elliptic Markov chain} with parameters $a$ and $b$. In the case where $a$ and $b$ are constant, and $M$ is Euclidean space, the elliptic Markov chain reduces to the example in Section 3 of \cite{georgiou2016anomalous}.

By choosing coordinates, we could write down multidimensional integrals for what Theorem \ref{thm:mainresultA} called $\underline{\nu}$ and $\overline{\nu}$. However, these integrals are somewhat complicated. Rather than attempt to evaluate them directly, we shall instead estimate them in terms of the second moments of $d_{\text{tot}}$ and $d_{\text{rad}}$. This will better enable comparison with the results in \cite{georgiou2016anomalous}.

We claim that
\begin{align*}
    \mathbb{E}_p[d_{\text{tot}}^2] &=  a(p)^2 + (d-1)b(p)^2, \\
    \mathbb{E}_p[d_{\text{rad}}^2] &= a(p)^2.
\end{align*}
To prove this, note that the computation in \cite[p. 7]{georgiou2016anomalous}, establishes this result when $V$ has the Euclidean inner product. The general result follows from the definition of $\xi_{V}$ together with the fact that any two inner product spaces of dimension $d$ whose inner products are positive definite are isometric.

For simplicity, we assume that both the chain and the underlying manifold are radially symmetric, meaning that the curvature tensor of $M$ at a point $p$ and the functions $a(p)$ and $b(p)$ depend only on the distance $r$ between $p$ and $O$. We further assume that there exists $\epsilon>0$ such that we have $a(r) \geq \epsilon$ for all $r$, and that $a$ and $b$ are bounded above. It then follows from Proposition \ref{prop:radialnondegeneracycriterion} that Assumption \ref{assumptionnotconfined} holds. Also, Assumption \ref{assumptiondtotmoments} holds, because we always have
\[ d_{\text{tot}} \leq d_{\max} := \sqrt{d} \max\left(\sup_{r \geq 0} a, \sup_{r \geq 0} b\right)\]
Further, define
\[ k_{\max}(r) = \sup \: \sqrt{\left|\mathcal{SC}(q, \Pi)\right|}, \: \: \: k_{\min}(r) = \inf \: \sqrt{\left|\mathcal{SC}(q, \Pi)\right|} \]
where $\mathcal{SC}$ is the sectional curvature and the suprema and infima are taken over all points $q \in M$ such that there exists $p \in S(r)$ with $\text{Dist}_M(p,q) \leq d_{\max}$ and all planes $\Pi \in T_q M$. We assume that $k_{\text{min}}$ and $k_{\text{max}}$ exist, are finite, and are everywhere strictly greater than zero. It follows, using Lemma \ref{lemma:fgrowthestimate}, that
\begin{align*}
    \underline{\nu}_1(r) &\leq \frac{1}{k_{\max}} \mathbb{E}_r [\log(\cosh(k_{\max} d_{\text{tot}})+\phi \sinh(k_{\max} d_{\text{tot}}))] \\
    &\leq \mathbb{E}_r[J_{\max}(k_{\max},d_{\text{tot}})(d_{\text{tot}}^2-d_{\text{rad}}^2)] \\
    &\leq J_{\max}(k_{\max},d_{\text{max}}) \mathbb{E}_r[d_{tot}^2-d_{rad}^2] \\
    &=J_{\max}(k_{\max},d_{\text{max}})(d-1)b^2(r)
\end{align*}
and similarly, $\overline{\nu}_1(r) \geq J_{\min}(k_{\min},d_{\max})(d-1)b^2(r)$.

We could estimate $\nu_2$ using Theorem \ref{thm:mainresultB}(ii), but it is simpler to observe that
\begin{equation*}
   \mathbb{E}_p[d_{\text{rad}}^2 \: 1_{d_{\text{rad}}>0}] \leq  \mathbb{E}_p[F(k,d_{\text{rad}},d_{\text{tot}})^2] \leq   \mathbb{E}_p[d_{\text{tot}}^2]
\end{equation*}
and hence, by symmetry,
\begin{equation*}
   \frac{1}{2}a^2(r) \leq \underline{\nu}_2(r) \leq \overline{\nu}_2(r) \leq  a^2(r) + (d-1)b^2(r)
\end{equation*}
It follows from Theorem \ref{thm:mainresultA} and the comparison theorem that
\begin{corollary}
\label{corol:ellipticexample}
The elliptic Markov chain with parameters $a(r)$ and $b(r)$ on a radially symmetric manifold is: 
\begin{enumerate}[(i)]
    \item Transient if \[ \lim_{r \rightarrow \infty}\left( 2r J_{\min}(k_{\min},d_{\max}) (d-1)b^2-a^2-(d-1)b^2 \right) > 0 \]
    \item Recurrent if there exists $\epsilon>0$ such that 
    \[ 2r J_{\max}(k_{\max},d_{\max}) (d-1)b^2 \leq \frac{1}{2}\left(1+\frac{1-\epsilon}{\log r} \right)a^2 \]
    for all sufficiently large $r$.
\end{enumerate}
\end{corollary}

A result of Azencott \cite{lenz2011random} shows that if $k_{min}(r) \geq C r^{2+\epsilon}$ for constants $C,\epsilon>0$ then $M$ is stochastically incomplete, so, by choosing $b(r)$ to decay fast enough that Corollary \ref{corol:ellipticexample}(ii) holds, we have found a recurrent chain on such a manifold, as promised.

Figure 2 gives shows numerical simulations of the hyperbolic elliptic random walk in dimension two for different choices of $a(r)$ and $b(r)$. In the first two examples, $a$ and $b$ are constant, and in the third $b(r) \rightarrow 0$ as $r \rightarrow \infty$. Only the third simulation shows recurrence, whilst if analogues of these chains were constructed in Euclidean space, both the second and the third would be recurrent.
%FIGURE ABOVE REFERENCED MANUALLY

\begin{figure}[]
\label{figure:threewalks}
\begin{center}
\includegraphics[width=10cm]{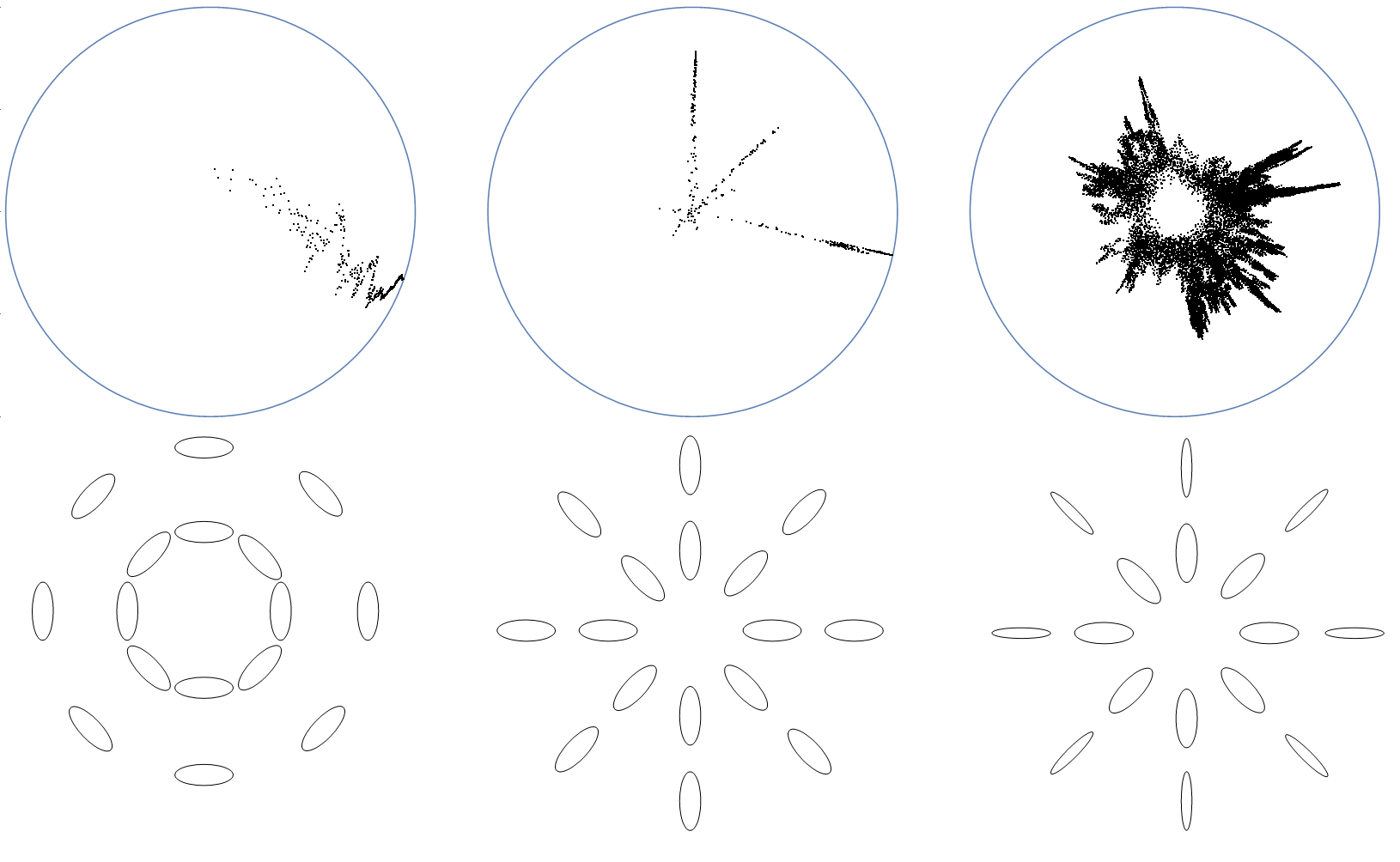}      
\end{center}
\caption{Upper row: Simulations of some elliptic Markov chains in the hyperbolic plane. Lower row: Schematic representations of these chains.}
\end{figure}

%----

\appendix
\section{Computing the Radial Increment}
In the computation that follows, we use the Lorentz model of hyperbolic space, which we now describe. Consider $\mathbb{R}^{d+1}$ with Cartesian coordinates $(x_0,x_1,\dots,x_d)$, and denote by $B$ the Minkowski bilinear form
\[ B(\textbf{x},\textbf{y})=-x_0 y_0 + x_1y_1+x_2y_2+\dots+x_dy_d \]
For $k>0$, let $\mathcal{H}_k$ be the hyperboloid
\[ \mathcal{H}_k=\{ \textbf{x} \in \mathbb{R}^{d+1} : B(\textbf{x},\textbf{x})=-\frac{1}{k^2} \text{ and } x_0>0  \} \]
Given a point $\textbf{x} \in \mathcal{H}_k$,  define the \textit{tangent space} $T_\textbf{x}\mathcal{H}_k$ by
\[ T_\textbf{x}\mathcal{H}_k = \{ \textbf{y} \in \mathbb{R}^{d+1}: B(\textbf{x},\textbf{y})=0 \} \]
This is a $d$-dimensional real vector space. Although $B$ is not positive definite, its restriction to $T_\textbf{x}\mathcal{H}_k$ is, and it can be shown that $B$ makes $\mathcal{H}_k$ into a complete Riemannian manifold of constant curvature $-k^2$. The Riemannian distance between $\textbf{x},\textbf{y} \in  \mathcal{H}_k$ is
\begin{equation} \label{lorentzdistance}
    \text{Dist}_{\mathcal{H}}(\textbf{x},\textbf{y})=\frac{1}{k} \text{arccosh}(-B(\textbf{x},\textbf{y})k^2)
\end{equation}
and the exponential map $\exp_\textbf{x}: T_\textbf{x}\mathcal{H}_k \rightarrow \mathcal{H}_k$ is given by
\begin{equation} \label{lorentzexp}
    \exp_{\textbf{x}}(\textbf{v})=\cosh(k||\textbf{v}||)\textbf{x}+\frac{\sinh(k||\textbf{v}||)}{k||\textbf{v}||}\textbf{v}
\end{equation}
where $||\textbf{v}||=\sqrt{B(\textbf{v},\textbf{v})}$, and \eqref{lorentzexp} is understood in the limiting sense if $\textbf{v}=\textbf{0}$. The facts stated above are well known (although they are usually stated only for $k=1$); see for example \cite{nickel2018learning}. 

\begin{proof}[Proof of Proposition \ref{prop:incrementexact}]
After applying an isometry, we may assume that $O=\left( \frac{1}{k},0,0,\dots,0 \right)$ is a particular point, and that $p_n$ lies on the half-geodesic $t \mapsto \gamma(t)$ given by
\[ \gamma(t) = \exp_O\left( t \: (0,1,0,\dots,0)^T\right) = \left(\frac{1}{k} \cosh(t k), \frac{1}{k} \sinh(t k), 0, \dots,0\right)^T,  \]
emanating from $O$. Note that $\gamma$, as defined above, is a unit-speed geodesic because $(0,1,0,\dots,0)^T$ is a unit-length vector in $T_O M$. Let $\tau>0$ be such that $p_n = \gamma(\tau)$. Then, in $T_{p_n}M$,
\[ \textbf{e}_{\text{rad}} = - \frac{d}{dt}\bigg|_{t=\tau} \gamma(t) = \left( -\sinh(\tau k), -\cosh(\tau k), 0, \dots, 0 \right)^T.  \]
Observe that $\mathcal{B}=\{\textbf{e}_{\text{rad}}, \textbf{e}_{2}, \dots, \textbf{e}_{n} \}$ is an orthonormal basis of $T_{p_n} M$, where $\textbf{e}_i$ is a vector equal to 1 at the $(j+1)^{\text{th}}$ place and 0 elsewhere. Let $\textbf{w}=\exp_{p_n}^{-1}(p_{n+1})$, and suppose that the representation of $\textbf{w}$ with respect to $\mathcal{B}$ is 
\[ \textbf{w} = h_{\text{rad}} \textbf{e}_{\text{rad}} + h_2 \textbf{e}_2 + \dots + h_n \textbf{e}_n \]
by orthonormality of $\mathcal{B}$, $d_\text{rad}=-h_{\text{rad}}$, and $d_{\text{tot}}=\sqrt{h_{\text{rad}}^2+h_2^2+\dots+h_n^2}$. It remains to compute the position of $p_{n+1}$, and hence its distance from $O$.
\begin{align*}
    p_{n+1} &= \exp_{p_n} \textbf{w} \\
    &= \cosh(k d_{\text{tot}}) \begin{bmatrix} k^{-1} \cosh(k\tau ) \\ k^{-1} \sinh(k \tau ) \\ 0 \\ \vdots \\ 0 \end{bmatrix} + \frac{\sinh k d_{\text{tot}}}{kd_{\text{tot}}} \begin{bmatrix} d_{\text{rad}}\sinh(k \tau ) \\ d_{\text{rad}}\cosh(k \tau ) \\ h_2 \\ \vdots \\ h_n \end{bmatrix}
\end{align*}
And therefore
\[ B(p_n, p_{n+1}) = -\frac{1}{k^2} \left( \cosh(kd_{\text{tot}} ) \cosh(k \tau) + \frac{d_{\text{rad}}}{d_{\text{tot}}} \sinh(kd_{\text{tot}} ) \sinh(k \tau) \right), \]
so applying \eqref{lorentzdistance} yields the result.
\end{proof}

\section{The Usual Definition of Recurrence}

We sketch how to modify the example in Section \ref{section:examples} to give a chain which we can prove is recurrent in the usual sense, meaning that if $N$ is an open neighbourhood of $M$ then $X_n \in N$ infinitely often almost surely. Our argument is heavily based on \cite[Example 2.3.20]{menshikov2016non}, which appeals to the following extension of the Borel--Cantelli lemmata, due to L\'{e}vy. 
\begin{theorem}
Let $(\mathcal{F}_n)_{n \in \mathbb{N}}$ be a filtration and $(E_n)_{n \in \mathbb{N}}$ a sequence of events with $E_n \in \mathcal{F}_n$. Then, upto sets of probability zero,
\begin{equation}
\label{eq:levy}
    \{ \omega \in \Omega: E_n \text{ infinitely often } \} 
    = \{ \omega \in \Omega: \sum \mathbb{P}[E_n \mid \mathcal{F}_{n-1}]=\infty \}
\end{equation}
\end{theorem}

The measures used in Section \ref{section:examples} are supported on elliptical shells in $\mathbb{R}^d$; for this section it is convenient to use solid shapes instead, so we proceed as follows. In each tangent space $T_p M$, extend $e_{\text{rad}}(p)$ to an orthonormal basis $\mathcal{B}(p)=\{e_{\text{rad}},e_{2},e_{3},\dots,e_{d} \}$. Let 
\[v_p = h_{1} e_{\text{rad}} + h_2 e_2 + \dots h_d e_d \]
be a random vector in $T_p M$ whose law is given by taking the $h_i$ to be independent and given by
\[ h_i \sim \begin{cases} \sqrt{3} \: \text{Uniform}[-a(r),a(r)] & \text{ if } i=1 \\ 
\sqrt{3} \: \text{Uniform}[-b(r),b(r)] & \text{ if } i=2,3,\dots,n \end{cases} \]
This gives a measure on each $T_p M$ and hence a Markov chain on $M$. Moreover, one can verify that $\mathbb{E}_p[d_{\text{tot}}^2]$ and $\mathbb{E}_p[d_{\text{rad}}]^2$ are the same as for the example in Section \ref{section:examples}. Let $a(r)$ and $b(r)$ be chosen so that \ref{corol:ellipticexample}(ii) holds. Then there is some neighbourhood $N_0$ of $O$ such that the chain visits $N_0$ infinitely often almost surely.

Now let $N$ be an arbitrary open neighbourhood of $M$. Consider the case where $N$ is not contained in $N_0$; the case where $N \subset N_0$ is similar. By shrinking $N$ if necessary, we may assume that $N$ is an open ball disjoint from $N_0$. For each point $x \in M$, the distribution of $X_{n+1}$ conditional on $X_{n}=x$, is supported on a compact set $S \subset M$, and, due to our choice of the distribution of the $h_i$, is dense in $S$. It follows that for each $x \in N_0$, there exist $m_x \in \mathbb{N}, \delta_x >0$ such that 
\[ \mathbb{P}[\tau_N \leq m \mid X_0 = x ] \geq \delta \]
Where $\tau_{N}:=\min\{n \in \mathbb{N}: X_n \in N \} $. Further, the boundedness of $N_0$ implies that $m$ and $\delta$ may be chosen so as not to depend on $x$. We have already shown that $X_n \in N_0$ i.o., and so we may define a sequence of stopping times $(\tau_n)_{n \in \mathbb{N}}$ by $\tau_1 = \tau_{N_0}$ and $\tau_{n+1} = \min \{ k \geq \tau_{n}+m : X_k \in N_0 \}$. In other words, $\tau_n$ is the $n^{th}$ return to $N_0$, except that we do not count returns to $N_0$ that are less than $m$ steps apart. Let $\mathcal{F}_n = \sigma(X_1,\dots,X_n), \mathcal{G}_n = \mathcal{F}_{\tau_n}$ and $E_n = \{ \min \{ n: X_n \in N \} \leq \tau_n \}$. Then $E_1 \subset E_2 \subset \dots$, and
\[ \mathbb{P}[E_{n+1} \mid \mathcal{G}_{n}] \geq \mathbb{P}[E_{n+1}(1_{E_n} + 1_{E_n^c} ) \mid \mathcal{G}_{n}] \geq 1_{E_n} + \delta \: 1_{E_n^c} \geq \min(1,\delta) \]  
Applying \eqref{eq:levy}, we see that $E_n$ i.o. almost surely. This proves that, regardless of where the chain is currently located, it will visit $N$ via $N_0$ in finite time. It will therefore do so infinitely often.

\bibliographystyle{plain}
\bibliography{references}

\end{document}